\documentclass[12pt,psamsfonts]{amsart}
\usepackage{amssymb,eucal,epsfig}
\RequirePackage{xy}
\xyoption{all}

\def\Z{\mathbb Z}

\theoremstyle{plain}

\newtheorem*{Morley}{Morley's Congruence}
\newtheorem*{Lehmer}{Lehmer's Congruence}
\newtheorem*{Sylvester}{Sylvester's Congruence}
\newtheorem*{Wohlstenholme1}{Wolstenholme's theorem}
\newtheorem*{Wohlstenholme2}{Wolstenholme's congruence}

\newtheorem*{lemma}{Lemma}
\newtheorem*{corollary}{Corollary}
\newtheorem*{remark}{Remark}

\theoremstyle{definition}

\begin{document}
\title{Sylvester's, Wolstenholme's, Morley's and Lehmer's congruence theorems revisited }

\author{Christian Aebi, Grant Cairns}


\address{Coll\`ege Calvin, Geneva, Switzerland 1211}
\email{christian.aebi@edu.ge.ch}
\address{Department of Mathematics, La Trobe University, Melbourne, Australia 3086}
\email{G.Cairns@latrobe.edu.au}

\maketitle

In commemoration of J. Wolstenholme's article published 150 years ago \cite[1862]{Wo}, we present here below elementary proofs of famous congruences of that period such as Sylvester's and Wolstenholme's and continue by deducing Morley's and Lehmer's congruences. Throughout this paper, largely inspired by \cite{CA},  $p$ always denotes an odd prime and by a slight abuse of notation, $\frac1i$ denotes the multiplicative inverse of $i$ modulo
$p$ or modulo $p^2$, according to the context. For convenience, we
introduce the following symbols:
\[
S=\sum_{0<i<p}\frac{1}{i} \quad , \quad S_a=\sum_{\substack{{0<i<p}\\{i\equiv a \pmod 2}}}\frac{1}{i}
\quad\text{and}\quad S_{ab}=\sum_{\substack{{0<i<j<p}\\{i\equiv a,
j \equiv b  \pmod 2}}}\frac{1}{ij}
\]

\section*{Sylvester's congruence}

Let $q=\frac{2^{p-1}-1}{p}$ be the Fermat quotient to base 2. We start by observing that  the binomial coefficient
$\binom{p}{k}=\frac{p}{k}\binom{p-1}{k-1}\equiv (-1)^{k-1}\frac{p}{k}\pmod{p^2}$ to deduce that
\[
2^p = (1+1)^p=2+\sum_{k=1}^{p-1}\binom{p}{k}\equiv 2+p(S_1 - S_0)\pmod{p^2}
\]
Subtracting 2, dividing by $p$ and then adding $0\equiv
S_0+S_1\pmod{p}$ on both sides of the equivalence
yields
$2q\equiv 2S_1\equiv -2S_0\pmod{p}$. So we have proved

\begin{Sylvester}\cite[1861]{Sy} If $p>2$ then $2S_0\equiv \sum_{0<i<\frac{p-1}{2}}\frac{1}{i}\equiv- 2q \pmod{p}\quad \verb"(SC)"$
\end{Sylvester}

\section*{Wolstenholme's congruence and theorem}

\begin{Wohlstenholme1}\cite[1862]{Wo}
If  $p>3$ then $S \equiv 0 \pmod {p^2}$.
\end{Wohlstenholme1}

\begin{proof} Working in $\Z_p$, in which $i^2\equiv(p-i)^2$ we obtain the following congruences
\small{\small{\begin{align*}
2S \equiv
2\sum_{i=1}^{\frac{p-1}{2}}\frac{1}{i}+\frac{1}{p-i} \equiv
2p\sum_{i=1}^{\frac{p-1}{2}}\frac{1}{i(p-i)} \equiv
2p\sum_{i=1}^{\frac{p-1}{2}}\frac{-1}{i^2} \equiv 
-2p\sum_{i=1}^{\frac{p-1}{2}}i^2 \equiv 
-p\sum_{i=1}^{p-1}i^2 \equiv 
-p\frac{(p-1)p(2p-1)}{6}
\end{align*}}}
\end{proof}

Before proving Wolstenholme's congruence we observe that:
\begin{remark}
\qquad $2\sum_{1\leq
i<j}^{p-1}\frac{1}{ij}=S^2-\sum_{i=1}^{p-1}\frac{1}{i^2}$
\end{remark}

\begin{Wohlstenholme2}\cite[1862]{Wo}
If $p>3$ then $\binom{2p-1}{p-1} \equiv1\pmod{p^3}$
\end{Wohlstenholme2}

\begin{proof}
We start by developing the binomial coefficient in $\Z_{p^3}$, then
use the previous Remark  and conclude with Wolstenholme's theorem.
\begin{align*}
\binom{2p-1}{p-1}&=(-1)^{p-1}\frac{(1-2p)(2-2p)\ldots((p-1)-2p)}{1\cdot 2 \ldots (p-1)}=\prod_{i=1}^{p-1}\left(1-\frac{2p}{i}\right)\\
&\equiv 1-2pS+4p^2\left(\sum_{1\leq
i<j}^{p-1}\frac{1}{ij}+\sum_{i=1}^{p-1}\frac{1}{i^2}\right)\\
&\equiv 1-2pS+4p^2S^2+2p^2\sum_{i=1}^{p-1}\frac{1}{i^2}\equiv
1\pmod{p^3}
\end{align*}
\end{proof}

\section*{Morley's congruence}
\begin{Morley}\cite[1895]{Morl} If $p>3$, then
$(-1)^{(p-1)/2} \left(\smallmatrix p-1\\
\frac{p-1}{2} \endsmallmatrix\right) \equiv 4^{p-1}
\pmod {p^3}.$
\end{Morley}
Notice the right side of the above congruence may also be written $(1+pq)^2$.
\begin{lemma}\label{L1}\
\begin{enumerate}
\item $S_0 \equiv -S_1 \pmod {p^2}$
\item $S_{0}^{2} \equiv -S_{01}-S_{10}\pmod {p}$
\item $S_{00} \equiv S_{11} \pmod {p}$
\item $2S_{00} \equiv -S_{01} \pmod {p}$
\end{enumerate}
\end{lemma}

\begin{proof}
(a) follows  from Wolstenholme's theorem, since $S_0 + S_1 =
\sum_{i=1}^{p-1}\frac{1}{i}\equiv 0 \pmod{p^2}$.\\
Concerning (b),\quad $S_{0}^2\equiv S_0\cdot(-S_1)\equiv -S_{01}-S_{10}\equiv -S_{01} \pmod{p}$\\
As regards (c),
\begin{equation*}
S_{00}\equiv \sum_{\substack{{0<i<j<p}\\{i\equiv 0, j \equiv 0 \pmod
2}}}\frac{(-1)(-1)}{(p-i)(p-j)}\equiv
\sum_{\substack{{0<l<k<p}\\
{l\equiv 1, k \equiv 1 \pmod 2}}}\frac{1}{lk}\equiv S_{11} \pmod{p}.
\end{equation*}
Finally for (d),
\begin{align*}
S_{00}&\equiv -\sum_{\substack{{0<i<j<p}\\{i\equiv 0, j \equiv 0
\pmod 2}}}\frac{1}{i(p-j)}\equiv
-\sum_{\substack{{0<i<i+k<p}\\{i\equiv 0, k \equiv 1 \pmod
2}}}\frac{1}{ik}\equiv -\sum_{\substack{{0<k<l<p}\\{k\equiv 1, l
\equiv 1 \pmod 2}}}\frac{1}{k(l-k)}\\
&\equiv -\sum_{\substack{{0<k<l<p}\\{j\equiv 1, k \equiv 1 \pmod
2}}}\left(\frac{1}{kl} + \frac{1}{(l-k)l}\right) \equiv
-S_{11}-\sum_{\substack{{0<m<l<p}\\{m \equiv 0, l \equiv 1 \pmod
2}}}\frac{1}{ml}\equiv -S_{11}-S_{01}\pmod{p}.
\end{align*}
\end{proof}
We now turn to the terms in Morley's theorem. Since

\begin{equation*}
2^p=\sum_{i=0}^p
\binom{p}{i}=2+\sum_{i=1}^{p-1}\frac{p}{i}\binom{p-1}{i-1}=2
+\sum_{i=1}^{p-1}\frac{p}{i}(-1)^{i-1}\left(1-\frac{p}{1}\right)
\left(1-\frac{p}{2}\right)\ldots\left(1-\frac{p}{i-1}\right)
\end{equation*}
we develop the product and apply (a) and (c)
\begin{align*}
2^{p-1}&\equiv 1-\frac{p}2\cdot\sum_{i=1}^{{p-1}} \frac{(-1)^{i}}i+
\frac{p^2}2\cdot\sum_{1\leqslant j< i<p} \frac
{(-1)^{i}}{ij}\\
&\equiv1-\frac{p}{2}(S_0-S_1)+\frac{p^2}{2}(S_{00}-S_{01}+S_{10}-S_{11})
\equiv 1 -pS_0 - \frac{p^2}{2}(S_{10}-S_{01})\pmod{p^3}.
\end{align*}
By squaring and applying (b) and (d) we obtain\\
\begin{align*}
4^{p-1}\equiv 1-2pS_0 + p^2(S^2_0-S_{01}+S_{10})&\equiv 1 -2pS_0+p^2(-S_{01}-S_{10}-S_{01}+S_{10} )\pmod{p^3}\\
&\equiv 1-2pS_0 + 4p^2S_{00} \pmod{p^3}
\end{align*}
On the other hand, the development of the middle binomial coefficient gives
\begin{align*}
(-1)^\frac{p-1}{2}\binom {p-1}{\frac{p-1}{2}}&\equiv
\left(1-\frac{p}{1}\right)\left(1-\frac{p}{2}\right)\ldots\left(1-\frac{p}{\frac{p-1}{2}}\right)\\
&\equiv 1-p\sum_{i=0}^{\frac{p-1}{2}}\frac{1}{i} + p^2\sum_{1\leq
j\leq i\leq \frac{p-1}{2}} \frac{1}{ij}\\
&\equiv 1 -2pS_0 + 4p^2S_{00} \pmod{p^3}
\end{align*}
which is the desired result.

\section*{Lehmer's congruence}
Using the above we prove that Morley's and Lehmer's congruences are equivalent.
\begin{Lehmer}\cite[1938
pg 358]{Lehmer}
$2S_0\equiv \sum_{i=1}^{\frac{p-1}{2}}\frac{1}{i}\equiv-2q+pq^2\pmod{p^2}$.
\end{Lehmer}

\begin{proof}
Notice that Morley's congruence can be written $(1+pq)^2\equiv 1 -2pS_0 + 4p^2S_{00} \pmod{p^3}$. Therefore, expanding the left side of the congruence, subtracting 1 and  dividing by $p$ gives $2q+pq^2\equiv-2S_0+4pS_{00}\pmod{p^2}$. Then by the Remark and \verb"SC" we find that
\[
4S_{00}\equiv \sum_{0\le i<j<\frac{p-1}{2}}\frac{1}{ij}\equiv \left(  \sum_{0\le i<\frac{p-1}{2}}\frac{1}{i}\right)^2 \equiv 2q^2 \pmod{p}
\]
which replaced above allows us to conclude.
\end{proof}

In view of our results we reconsider the previous Lemma and obtain finally

\begin{corollary}
\begin{enumerate}
\item $S_0 \equiv -S_1 \equiv -q+\frac{1}{2}pq^2\pmod {p^2}$
\item $S_{00} \equiv S_{11}\equiv \frac{1}{2}q^2 \pmod {p}$
\item $S_{01}\equiv -2S_{00}  \equiv -q^2\pmod {p}$
\item $S_{10} \equiv -S_{01}-S_{0}^{2}\equiv 0\pmod {p}$
\end{enumerate}
\end{corollary}

\textbf{Acknowledgments}. We address our thanks to the anonymous
 referee who restructured our initial proof of Morley's congruence theorem \cite{CA} by introducing the $S$ notation. 

\bibliographystyle{amsplain}

\end{document}